\documentclass[10pt]{amsart}

\usepackage{amscd,amsmath,amsfonts,amsthm,amssymb}
\usepackage{enumitem}
\usepackage{graphicx}
\usepackage{xcolor}
\usepackage[all]{xy}
\usepackage[top=25truemm,bottom=25truemm,left=15truemm,right=15truemm]{geometry}

\usepackage[colorlinks=true,citecolor=blue,linkcolor=magenta,urlcolor=green,
  bookmarksnumbered=true,pdfborder={1 0 0},bookmarkstype=toc]{hyperref}

\usepackage[nameinlink,noabbrev]{cleveref}

\theoremstyle{definition}
\newtheorem{thm}{Theorem}[section]
\newtheorem{dfn}[thm]{Definition}
\newtheorem{lem}[thm]{Lemma}
\newtheorem{pro}[thm]{Proposition}

\newtheorem{exa}[thm]{Example}
\newtheorem{rmk}[thm]{Remark}
\newtheorem{nota}[thm]{Notation}

\crefname{thm}{Theorem}{Theorems}
\crefname{lem}{Lemma}{Lemmas}
\crefname{pro}{Proposition}{Propositions}
\crefname{cor}{Corollary}{Corollaries}
\crefname{dfn}{Definition}{Definitions}
\crefname{exa}{Example}{Examples}
\crefname{rmk}{Remark}{Remarks}
\crefname{nota}{Notation}{Notations}

\newcommand{\exend}{\unskip\nobreak\hfill$\blacklozenge$}

\newcommand{\D}{\mathcal{D}}
\newcommand{\F}{\mathbb{F}}
\newcommand{\M}{\mathcal{M}}
\newcommand{\Q}{\mathbb{Q}}
\newcommand{\R}{\mathbb{R}}
\newcommand{\Z}{\mathbb{Z}}

\renewcommand{\phi}{\varphi}
\renewcommand{\epsilon}{\varepsilon}

\newcommand{\Hom}{\mathop{\mathrm{Hom}}\nolimits}

\newcommand{\End}{\mathop{\mathrm{End}}\nolimits}

\renewcommand{\ker}{\mathrm{ker}}

\newcommand{\tr}{\mathrm{tr}}
\newcommand{\Tr}{\mathrm{Tr}}

\renewcommand{\lim}[1][]{\mathop{\mathrm{lim}}\limits_{#1}}
\newcommand{\colim}[1][]{\mathop{\mathrm{colim}}\limits_{#1}}

\newcommand{\op}{\text{op}}
\newcommand{\st}{\text{st}}
\newcommand{\arc}{\text{arc}}
\newcommand{\mot}{\text{mot}}

\newcommand{\Mod}{\mathrm{Mod}}
\newcommand{\Ban}{\mathrm{Ban}}
\newcommand{\Ani}{\mathrm{Ani}}

\newcommand\bcdot{\ensuremath{%
  \mathchoice%
   {\mskip\thinmuskip\lower0.2ex\hbox{\scalebox{1.5}{$\cdot$}}\mskip\thinmuskip}}%
   {\mskip\thinmuskip\lower0.2ex\hbox{\scalebox{1.5}{$\cdot$}}\mskip\thinmuskip}%
   {\lower0.3ex\hbox{\scalebox{1.2}{$\cdot$}}}%
   {\lower0.3ex\hbox{\scalebox{1.2}{$\cdot$}}}%
   }

\begin{document}

\title[Notes on congruence zeta functions via a Berkovich approach]{Notes on congruence zeta functions via a Berkovich approach}
\author[Y. Yamada]{Yuto Yamada}
\address{Department of Mathematics, Institute of Science Tokyo, 2-12-1 Ookayama, Meguro, Tokyo 152-8551}
\email{yamada.y.f243@m.isct.ac.jp}
\begin{abstract}
We revisit congruence zeta functions of smooth projective varieties over finite fields in the framework of Scholze’s Berkovich motives. To use it, we define the notion of motivic congruence zeta functions via categorical traces as an analogue of Kahn's approach. With this motivic setting, we prove that our new zeta functions agree with the classical one.
\end{abstract}


\date{\today}

\maketitle

\tableofcontents

\section{Introduction}

\subsection*{A Brief History of Congruence Zeta Functions}

The congruence zeta function of a variety over a finite field is one of the most classical and important objects in arithmetic geometry. Let $X$ be a variety over a finite field $\F_q$ (i.e. separated $\F_q$-scheme of finite type). Its congruence zeta function is defined by
\[
Z(X,t)=\exp\left(\sum_{n=1}^\infty\dfrac{\#X(\mathbb{F}_{q^n})}{n}t^n\right)\in 1 + t\Q[[t]].
\]
For smooth projective curves (i.e. algebraic curves, in particular, elliptic curves), Hasse and Weil showed that $Z(X,t)$ is a rational function and identified its numerator as the characteristic polynomial of Frobenius acting on the Jacobian. Grothendieck et al. have developed the theory of $\ell$-adic \'{e}tale cohomology, and proved the Grothendieck-Lefschetz trace formula (\cite{SGA}) for higher dimensional smooth projective varieties, which have contributed to the progression of Weil's conjecture (\cite{W}) and algebraic/arithmetic geometry.

\subsection*{A Motivic Approach}

Kahn has defined categorical zeta functions via a “motivic" interpretation of an ($\ell$-adic) \'{e}tale cohomology theory, and compared with the classical (congruence) zeta function in \cite{Ka}. His strategy is to use general (categorical) traces of triangulated categories and (classical)  six-functor formalism. In fact, he has challenged not only congruence zeta functions over a finite field, but also $L$-functions over a global field by the machinery of traces and six-functor formalism above. Similarly, our strategy can be also expected to extend to the cases of global fields.

\subsection*{Berkovich Motives}

On the other hand, Scholze has recently constructed a motivic six-functor formalism on small arc-stacks associated with Banach rings / Berkovich spaces in \cite{S}. This theory gives a unified construction of “motives" of geometric points, in particular, it recovers Voevodsky's \'{e}tale motives over algebraically closed fields (see also \cite[Proposition 1.14]{S}). Note that, over a discrete ground field, we can consider varieties as small arc-stacks via trivial norms.

\subsection*{Main Theorem}

The aim of this paper is to give a new description of congruence zeta functions of smooth projective varieties over finite fields. We can view the $\infty$-category $\widetilde{\D_\mot}(X)$ of “modified" motivic sheaves as a dualizable object in the $(\infty,2)$-category $\Mod_{\D_\mot(k)}(\mathrm{Pr}^L)$ where $X$ is a smooth projective variety over $k$. Via categorical trace formalisms, we define a \emph{motivic congruence zeta function} of $X$ in the spirit of Kahn's construction (,however, without explicitly using Chow motives). Our main result identifies this new motivic congruence zeta function with the classical one defined above:

\begin{thm}[\cref{thm:classical}]
For smooth projective variety $X$ over a finite field, we obtain $Z_\mot(X,t)=Z(X,t)$.
\end{thm}

\subsection*{Structure of this paper}

In \cref{sec:Recollection of Berkovich motives}, we recall the basic notions of Berkovich spaces and motivic sheaves according to \cite{S}. Since Berkovich motives include still quite new notions, we allow ourselves a slightly redundant recollection (e.g. Banach $\mathbb{C}$-algebras via Gelfand duality).

In \cref{sec:Quick Review of Categorical Trace}, we review the formalism of categorical traces in a presentably symmetric monoidal stable $\infty$-category and in its $(\infty,2)$-category of module categories. Our strategy is inspired by Clausen-Scholze's lecture notes (\cite{CC}) (and more general discussion by \cite{HSS}). 

Finally, in \cref{sec:Some Properties of Categorical Zeta Functions}, we define a “Berkovich motivic" version of congruence zeta functions for smooth projective varieties over a finite field, and we analyze our new motivic congruence zeta function via the categorical trace method and the Grothendieck--Lefschetz trace formula. Moreover, we want to emphasize that the same strategy can be applied, with little change, to suitable “varieties" over algebraically closed Banach fields.

\subsection*{Notation}

In this paper, we fix a prime number $p>0$. We assume that rings are commutative rings with unit. Basically, we will work in (stable) $\infty$-categories. 

\subsection*{Acknowledgement}

This work was supported by JST SPRING, Japan Grant Number JPMJSP2180.

\section{Recollection of Berkovich motives}\label{sec:Recollection of Berkovich motives}

In this section, we will review the notions of Berkovich spaces and motivic sheaves.

\begin{dfn}[{\cite[Definition 2.1,2.2,2.10]{S}}, Banach rings/fields]
A ring $R$ is \emph{Banach} if $R$ is equipped with a map $|-|_R:R\to\R_{\geq0}$ satisfying the following conditions:
\begin{enumerate}
    \item We have $|0|_R=0$ and $|-1|_R\leq1$.
    \item For any elements $a,b\in R$, we have $|ab|_R\leq|a|_R|b|_R$.
    \item For any elements $a,b\in R$, we have $|a+b|_R\leq|a|_R+|b|_R$.
    \item The ring $R$ is complete with respect to the induced topology by $|-|_R$.
\end{enumerate}
Moreover, a Banach ring $(K,|-|_K)$ is a \emph{Banach field} if $K$ is a field, and for any elements $a,b\in K$, we have $|ab|_K=|a|_K|b|_K$.

Furthermore, we let $\Ban$ denote the category of Banach rings with non-expanding maps (i.e. maps $f:R\to S$ of rings satisfying $|f(a)|_S\leq|a|_R$ for any element $a\in R$). Note that $\Ban$ has all colimits (\cite[Proposition 2.4]{S}).
\exend
\end{dfn}

\begin{dfn}[{\cite[Definition 2.13]{S}}, Berkovich spectra]
The \emph{Berkovich spectrum} of a Banach ring $R$ is the closed subspace $\M(R)\subset\prod_{a\in R}[0,|a|_R]$ of maps $\|-\|:R\to\R_{\geq0}$ satisfying the following conditions:
\begin{enumerate}
    \item We have $\|0\|=0$ and $\|1\|=1$.
    \item For any element $a\in R$, we have $\|a\|\leq|a|_R$. 
    \item For any elements $a,b\in R$, we have $\|ab\|=\|a\|\|b\|$.
    \item For any elements $a,b\in R$, we have $\|a+b\|\leq\|a\|+\|b\|$.
\end{enumerate}
For any point $x\in\M(R)$ (which corresponds to the map $\|-\|_x$), we let $K(x)$ denote the completion of the fraction field of the Banach ring $(R/\ker(\|-\|_x),\|-\|_x)$, which is a Banach field.
\exend
\end{dfn}

\begin{dfn}[(strictly) totally disconnectedness]
Let $(R,|-|_R)$ be a Banach ring. We call $R$:
\begin{enumerate}
    \item (\cite[Definition 2.17]{S}) \emph{analytic} if for any point $x\in\M(R)$, the corresponding Banach field $K(x)$ is non-discrete.
    \item (\cite[Definition 2.20 and Proposition 2.21]{S}) \emph{uniform} if for any element $a\in R$, we have $|a|_R=\sup\{\|a\|_x\mid x\in\M(R)\}=\lim[n\to\infty]|a^n|_R^{1/n}$; equivalently, $|-|_R$ is power-multiplicative.
    \item (\cite[Definition 3.10]{S}) \emph{totally disconnected} if $R$ is analytic and uniform, and $\M(R)$ is profinite (i.e. totally disconnected and compact Hausdorff).
    \item (\cite[Definition 3.10]{S}) \emph{strictly totally disconnected} if $R$ is totally disconnected, and for any point $x\in\M(R)$, the corresponding Banach field $K(x)$ is algebraically closed.
\end{enumerate}
Moreover, we let $\mathrm{TD}^\st$ denote the full subcategory of $\Ban$ spanned by strictly totally disconnected Banach rings.
\exend
\end{dfn}

We define the “arc-topology", the useful topology to analyze the “motives".

\begin{dfn}[arc-site]
We define the following notions:
\begin{enumerate}
    \item (\cite[Definition 3.1]{S}) A family $\{R\to S_i\}_{i\in I}$ of maps in $\Ban$ is an \emph{arc-cover} if there is a finite subset $J\subset I$ such that the induced map $\bigsqcup_{j\in J}\M(S_j)\to\M(R)$ is surjective.
    \item (\cite[Proposition 3.3]{S}) the site $(\Ban^\op,\arc)$ with the (finitary) Grothendieck topology given by arc-covers is called an \emph{arc-site}. Note that $(\Ban^\op,\arc)$ has a basis $\mathrm{TD}^\st$ (\cite[Corollary 3.12]{S}), and $(\mathrm{TD}^\st)^\op$ is subcanonical (\cite[Theorem 3.13]{S}).
    \item Let $R$ be a Banach ring. The arc-sheaf $\M_\arc(R)$ is the arc-sheafification of $S\longmapsto\Hom_{\Ban}(R,S)$.
    \item A sheaf of $(\Ban^\op,\arc)$ with values in $\Ani$ (i.e. an object in $\mathrm{Shv}_\arc(\Ban^\op,\Ani)$) is called an \emph{arc-stack}. Moreover, an arc-stack is \emph{small} if it can be represented by a \emph{small} colimit $\colim(\M_\arc(R))$.
\end{enumerate}
\exend
\end{dfn}

Firstly, we define some notions of arc-sheaves.

\begin{dfn}[finitary/ball-invariant/(effective) motivic sheaves]
Let $X$ be a small arc-stack. We define the following notions and $\infty$-categories:
\begin{enumerate}
    \item (\cite[Definition 4.10 and Proposition 4.12]{S}) The $\infty$-category $\D_\text{fin}(X)$ of \emph{finitary sheaves} on $X$ is defined as the full $\infty$-subcategory of $\mathrm{Fun}(\mathrm{TD}_{/X}^\st,\D(\Z))$ spanned by the functors which commute with finite products and filtered colimits; equivalently, the full $\infty$-subcategory of $\D(X_\arc):=\mathrm{Shv}_\arc(\Ban_{/X}^\op,\D(\Z))$ spanned by the arc-sheaves whose restriction to strictly totally disconnected Banach rings over $X$ commute with filtered colimits (see also the first assertion of \cite[Theorem 4.1]{S}).
    \item (\cite[Definition 5.1]{S}) An object $F\in\D(X_\arc)$ is \emph{ball-invariant} if for any Banach ring $R$ over $X$, the value on $R$ and $R\langle T\rangle_1$ agree (where $R\langle T\rangle_1$ denotes the free uniform Banach $R$-algebra on a variable $T$ with $|T|\leq1$).
    \item (\cite[Definition 5.2 and Proposition 5.8]{S}) The $\infty$-category $\D_\mot^\text{eff}(X)$ of \emph{effective motivic sheaves} on $X$ is defined as the full $\infty$-subcategory of $\D_\text{fin}(X)$ spanned by ball-invariant sheaves; equivalently, the essential image of $L:\D_\text{fin}(X)\to\D_\text{fin}(X)$ defined by
    \[
    K\longmapsto(R\longmapsto\colim[n\in\Delta^\op]K(R\langle \Delta^n\rangle)),
    \]
    where $R\langle \Delta^n\rangle$ denotes the Banach $R$-algebra $R\langle T_1,\ldots,T_n\rangle_{1,\ldots,1}$ adjoined $n$-variables freely.
    \item (\cite[Definition 5.18]{S}) The \emph{Tate twist} $\Z(1)$ is defined by the arc-sheaf $(\Z_\mot[\mathbb{G}_m]/\Z_\mot[\ast])[-1]$ where $\Z_\mot[-]$ denotes the free motivic sheaf discussed in \cite[Proposition 5.12 and Theorem 6.1]{S}. Note that if we work with nonarchimedean world, then it is equivalent to the arc-sheaf $(\mathbb{G}_m/(1+\mathcal{O}_{<1}))[-1]$ where $\mathcal{O}_{<1}$ denotes the sheaf $R\longmapsto R_{<1}$ (\cite[Definition 5.15 and Proposition 5.17]{S}); or if we work over $\mathbb{C}$, then it is equal to $\Z$.
    \item (\cite[Definition 9.1]{S}) The $\infty$-category $\D_\mot(X)$ of \emph{motivic sheaves} is defined by $\D_\mot^\text{eff}(X)[\Z(1)^{\otimes -1}]$. Note that for any negative integer $n\in\Z_{<0}$, $\Z(-n)$ can be well-defined as $(\Z(1)^{\otimes -1})^{\otimes(-n)}$ in the motivic sheaves.
\end{enumerate}
\exend
\end{dfn}

\begin{exa}[Banach $\mathbb{C}$-algebras]\label{exa:Betti}
By Gelfand duality (\cite{G}), the arc-site on Banach $\mathbb{C}$-algebras is Morita-equivalent to the site of compact Hausdorff spaces equipped with finitary jointly surjections (see also \cite[Proposition 3.9]{S}). Moreover, by \cite[Theorem 4.14]{S}, for any Banach $\mathbb{C}$-algebra $R$, there is an equivalence $\widehat{\D}(\M(R))\simeq\D_\text{fin}(\M_\arc(R))$ of $\infty$-categories (where $\widehat{\D}(\M(R))$ denotes the left-completion of $\D(\M(R))$) (see also \cite[Section 4]{C}, in particular, \emph{the topological case}).
\exend
\end{exa}

\begin{rmk}[six-functor formalism]
The following statement makes the construction of motivic sheaves reasonable: We can construct six-functor formalism $X\longmapsto\D_\mot(X)$ from the class of finite cohomological dimension (See also \cite[Theorem 3.4.11]{HM}). In this formalism, we can obtain the desired push-forward by \cite[Theorem 9.2, Corollary 10.2]{S}, any \'{e}tale map is cohomologically \'{e}tale by {\cite[Proposition 9.5]{S}}, and any “smooth" map is cohomologically smooth by {\cite[Corollary 9.6]{S}}.
\exend
\end{rmk}

\begin{exa}[examples for geometric points]\label{exa:geometric points}
If $K$ is an algebraically closed Banach field, there are the following explicit description of $\D_\mot(K)$:
\begin{enumerate}
    \item Similarly as in \cref{exa:Betti}, if $K=\mathbb{C}$, then $\D_\mot(K)$ is equal to $\D(\Z)$.
    \item (\cite[Proposition 10.1]{S}) If $K$ is non-discrete and mixed characteristic, then $\D_\mot(K)$ is compactly generated, the unit is compact, and all compact objects are dualizable.
    \item (\cite[Proposition 10.1]{S}) If $K$ is non-discrete and equal characteristic, then $\D_\mot(K)$ is compactly generated, the unit is compact, and all compact objects are dualizable. Moreover, with a splitting $k\to K$ of its residue field $k$, we can take
    \[
    \{\Z_\mot[X_K](-j)\mid\text{$X$ is a smooth projective variety over $k$, and $j$ is a nonnegative integer.}\}
    \]
    as a generating family of compact objects.
   \item (\cite[Theorem 11.1]{S}) If $K$ is discrete, then $\D_\mot(K)$ is generated by compact dualizable objects $\Z_\mot[X](-j)$ where $X$ is a smooth projective variety over $k$, and $j$ is an integer.
\end{enumerate}
\exend
\end{exa}

\begin{rmk}[$\ell$-adic realization]\label{rmk:realization}
To focus on the case of interest, say $K=\overline{\F}_{p^s}$ for some finite field $\F_{p^s}$, \cite[Theorem 11.1]{S} gives a very explicit and classical description of $R\Hom_{\D_\mot(K)}(\Z_\mot[X],\Z)$: after prime-to-$p$ profinite completion and tensoring ${}\widetilde{\otimes}\Q_\ell$ (where $\ell$ is a prime number coprime to $p$), it is equivalent to $R\Gamma(X_{\text{pro\'{e}t}},\Q_\ell)=R\Gamma(X_{\text{\'{e}t}},\Q_\ell)$. This construction is functorial, which, in particular, implies that, for any map $f:X\to Y$ of smooth projective varieties over $K$, the induced action $f^\ast$ on $R\Hom_{\D_\mot(K)}(\Z_\mot[Y],\Z)$ corresponds to the \'{e}tale one.
\exend
\end{rmk}

\section{Quick Review of Categorical Trace}\label{sec:Quick Review of Categorical Trace}

In this section, we will review the notion of categorical trace according to \cite[Lecture X\hspace{-1.2pt}I\hspace{-1.2pt}V and X\hspace{-1.2pt}V]{CC} (, and \cite[Section 2,3]{HSS} including more general and detailed discussions of categorical traces). 

\begin{nota}
We fix a presentably symmetric monoidal stable $\infty$-category $C$, and let $M$ denote the $(\infty,2)$-category $\Mod_C(\mathrm{Pr}^L)$ with unit $1$.
\exend
\end{nota}

\begin{dfn}[trace functors]
Let $X\in M$ be a dualizable object (i.e. there exists some object $X^\lor$ such that it admits a coevaluation map $\mathrm{coev}_X:1\to X\otimes X^\lor$ and an evaluation map $\mathrm{ev}_X:X^\lor\otimes X\to 1$ satisfying some compatibility (\cite[Subsection 2.1]{HSS})). For an endomorphism $f:X\to X$, we define the \emph{trace functor} $\tr(f|X)$ of $f$ as the following endomorphism of $1$ in $M$:
\[
1\xrightarrow{\mathrm{coev}_X}X\otimes X^\lor\xrightarrow{f\otimes\mathrm{id}}X\otimes X^\lor\simeq X^\lor\otimes X\xrightarrow{\mathrm{ev}_X}1.
\]
This construction gives a functor $\tr(-|X):\End_M(X)\to\End_M(1)\simeq C$. Note that this also defines a functor $\tr(-,-):\End(M)\to C$ in the sense of \cite[Definition 2.2]{HSS}.
\exend
\end{dfn}

\begin{pro}[properties of $\tr(-|-)$]\label{pro:properties}
Trace functors satisfy the following properties:
\begin{enumerate}
    \item For any dualizable objects $X,Y\in M$ and endomorphisms $f:X\to X, g:Y\to Y$, we have $\tr(f\otimes g|X\otimes Y)\simeq\tr(f|X)\otimes\tr(g|Y)$.
    \item For any localization sequence $(X,f)\to(Y,g)\to(Z,h)$ in the sense of \cite[Definition 3.2]{HSS}, we have a cofiber sequence $\tr(f|X)\to\tr(g|Y)\to\tr(h|Z)$.
\end{enumerate}
\end{pro}
\begin{proof}
The property (1) follows from this functorial construction. The property (2) is precisely \cite[Theorem 3.4]{HSS}.
\end{proof}

\begin{rmk}[formal discussion of $\#_f$]\label{rmk:formal discussion}
Let $f$ be an endomorphism of $C$. For any dualizable object $X\in M$, we let $f_X$ denote the endomorphism of $X$ which is the base-change of $f$. For any map $F:X\to Y$ of dualizable objects in $M$ with right adjoint functor $G:Y\to X$, there exists the following map:
\[
\tr(f|F):\tr(f_X|X)\to\tr(f_XGF|X)\simeq\tr(Ff_XG|Y)\to\tr(f_YFG|Y)\to\tr(f_Y|Y),
\]
where the first and last maps are induced by an adjoint, the second isomorphism follows from {\cite[Lecture X\hspace{-1pt}V]{CC}}, and the third map is induced by a natural transform $Ff_X\Rightarrow f_YG$ (see also the discussion in \cite[Lecture X\hspace{-1.2pt}V]{CC}). As an example, we apply to $Y=C$. For any dualizabble object $P$ in $X$, it gives the functor $-\otimes_CP:C\to X$ with $C$-linear right adjoint functor. The above discussion defines the map $\tr(f|C)\to\tr(f_X|X)$.

Finally, we can define the element $\#_{f,X}(P)\in\pi_0(\tr(f_X|X))$ as the image of $1\in\pi_0(\tr(f|C))$ by $\pi_0(\tr(f|C))\to\pi_0(\tr(f_X|X))$. Note that this construction is functorial, there is a map $\#_{f,X}: K_0(X)\to\pi_0(\tr(f_X|X))$. In fact, this map is “additive” (i.e. it sends a cofiber sequence $P''\to P\to P'$ in $X$ to $\#_{f,X}(P)=\#_{f,X}(P')+\#_{f,X}(P'')$) by \cref{pro:properties}(3) and (an analogue of) abstract discussion as in \cite[Proposition 15.1]{CC}.

Moreover, if a map $F:X\to Y$ of dualizable objects in $M$ admits a right adjoint, then the formal proof of \cite[Proposition 15.4]{CC} gives the following commutative diagram:
\[
\xymatrix{
{K_0(X)}\ar[r]^-{\#_{f,X}}\ar[d]_-{K_0(F)}&{\pi_0(\tr(f_X|X))}\ar[d]^-{\pi_0(\tr(f|F))} \\
{K_0(Y)}\ar[r]_-{\#_{f,Y}}&{\pi_0(\tr(f_Y|Y)),}
}
\]
where $K_0(X)$ denotes the $K_0$-group of the full $\infty$-subcategory spanned by dualizable objects in $X$.
\exend
\end{rmk}

\section{Some Properties of Categorical Zeta Functions}\label{sec:Some Properties of Categorical Zeta Functions}

In this section, we will specialize in the case of algebraic varieties (in motivic sheaves), and analyze categorical zeta functions.

\begin{lem}[dualizability]\label{lem:dualizability}
For a qcqs analytic space $X$ over an algebraically closed Banach field $K$ (in the sense of {\cite[Chapter 3]{B}}), we let $\widetilde{\D_\mot^\text{eff}}(X_\arc)$ (resp. $\widetilde{\D_\mot}(X_\arc)$) denote the (localizing) full $\infty$-subcategory generated by $\Z_\mot[U]$ for all open subsets $U\subset\mathbb{A}_X^n$ (resp. the presentable stable $\infty$-category $\widetilde{\D_\mot^\text{eff}}(X_\arc)[\Z(1)^{\otimes(-1)}]$) (, noting that $\Z(1)=(\Z_\mot[\mathbb{G}_m]/\Z_\mot[\ast])[-1]$ lies in $\widetilde{\D_\mot^\text{eff}}(X_\arc)$). The full $\infty$-subcategory $\widetilde{\D_\mot}(X_\arc)$ of $\D_\mot(X_\arc)$ is dualizable (as an object in $\Mod_{\D_\mot(K)}(\mathrm{Pr}^L)$).
\end{lem}
\begin{proof}
As in the proof of {\cite[Proposition 10.3]{S}}, every open subset $V\subset\mathbb{A}_X^n$ can be written as a filtered union of open subsets $U\Subset V$, so $\Z_\mot[V]$ can be rewritten as a filtered colimit with compact transition maps. Since $\widetilde{\D_\mot^\mathrm{eff}}(X_\arc)$ is precisely the $\infty$-subcategory $\mathcal C$ appearing in the proof of {\cite[Proposition 10.3]{S}}, the same argument shows that it is compactly assembled (i.e. dualizable), which gives dualizability of $\widetilde{\D_\mot}(X_\arc)$. In particular, by \cref{exa:geometric points}, $\widetilde{\D_\mot}(K)$ is equivalent to $\D_\mot(K)$, and it is rigid. The construction is compatible with the \(\D_\mot(K)\)-module structure
induced by the structural morphism \(X\to \operatorname{Spec}(K)\), which gives the desired dualizability over $\D_\mot(K)$.
\end{proof}

Therefore, we can apply the previous section to the motivic sheaves on “varieties", since the analytification of any variety over an algebraically closed field is qcqs.

\begin{nota}
For the rest of this section, we work over a finite field $\F_q$ with $q=p^s$, and set $k:=\overline{\F}_q$. For simplicity, we abbreviate $\D_\mot(k)$ as $k_\arc$, for example, $\tr(f|\D_\mot(k))$ as $\tr(f|k_\arc)$.
\exend
\end{nota}

\begin{dfn}[local term]
Let $X$ be a smooth projective variety over $\F_q$. We let $X_k$ (resp. $F:\mathrm{Spec}(k)\to\mathrm{Spec}(k)$) denote the base-change of $X$ along $\F_q\to k$ (resp. the induced map by the \emph{geometric} Frobenius map of $\mathrm{Spec}(\F_q)$). Note that $X_k$ can be considered a proper analytic space over an algebraically closed (discrete) Banach field $k$, the $\infty$-category $\widetilde{\D_\mot}(X_{k,\arc})$ of motivic sheaves is dualizable in $\D_\mot(k)$. We let
\[
\#_{F,n}:K_0(X_{k,\arc})\to\pi_0(\tr((F^\ast)^n|k_\arc))
\]
denote $(\pi_0(\tr((F^\ast)^n|f_\ast(-)^\lor\widetilde{\otimes}\Q_\ell))\circ\#_{(F_{X_k}^\ast)^n,X_{k,\arc}})(-)=(\#_{(F^\ast)^n,k_\arc}\circ K_0(f_\ast(-)^\lor\widetilde{\otimes}\Q_\ell))(-)$ as in \cref{rmk:formal discussion}, where $f$ denotes the structure map $f:X_k\to\mathrm{Spec}(k)$.
\exend
\end{dfn}

\begin{rmk}[computation of local term]\label{rmk:computation of local term}
For a smooth projective variety $X$ over $\F_q$ with base-change $X_k=X\times_{\mathrm{Spec}(\F_q)}\mathrm{Spec}(k)$, and a positive integer $n\in\Z_{>0}$, by construction and \cref{rmk:realization}, the $n$-th term $\#_{F,n}(1_{X_k})$ is equal to
\[
\#_{(F^\ast)^n,k_\arc}(R\Hom(\Z_\mot[X_k],\Z)\widetilde{\otimes}\Q_\ell)=\#_{(F^\ast)^n,k_\arc}(R\Gamma(X_{k,\text{\'{e}t}},\Q_\ell)).
\]
Also, we can compute $\#_{(F^\ast)^n,k_\arc}(P)$ as $\tr((F^\ast)^n|P)$ (where $P$ denotes $R\Gamma(X_{k,\text{\'{e}t}},\Q_\ell)$), noting that it is defined by the image of $1\in\pi_0(\tr((F^\ast)^n|k_\arc))$ by $\pi_0(\tr((F^\ast)^n|k_\arc))\to\pi_0(\tr((F^\ast)^n|k_\arc))$ (induced by $(-)\otimes P$) as in \cref{rmk:formal discussion}. The \emph{usual} trace $\tr((F^\ast)^n|P)$ can be given by $\sum_{i=0}^\infty(-1)^i\Tr((F^\ast)^n|H^i(P))$, which gives the following computation:
\[
\#_{F,n}(1_{X_{k,\arc}})=\#_{(F^\ast)^n,k_\arc}(R\Gamma(X_{k,\text{\'{e}t}},\Q_\ell))=\sum_{i=0}^\infty(-1)^i\Tr((F_{X_k}^\ast)^n|H^i(X_{k,\text{\'{e}t}},\Q_\ell)).
\]
In particular, $\#_{F,n}(1_{X_{k,\arc}})$ lies in $\Q$ for all positive integer $n\in\Z_{>0}$.
\exend
\end{rmk}

\begin{dfn}[motivic congruence zeta function]
Let $X$ be a smooth projective variety over $\F_q$. We let $X_k$ (resp. $F:\mathrm{Spec}(k)\to\mathrm{Spec}(k)$) denote the base-change of $X$ along $\F_q\to k$ (resp. the induced map by the \emph{geometric} Frobenius map of $\mathrm{Spec}(\F_q)$). The \emph{motivic congruence zeta function} $Z_\mot(X,t)$ is defined by 
\[
\exp\left(\sum_{n=1}^\infty\dfrac{\#_{F,n}(1_{X_{k,\arc}})}{n}t^n\right)\in\Q[[t]].
\]
\exend
\end{dfn}

\begin{thm}[classical congruence zeta functions]\label{thm:classical}
For a smooth projective variety $X$ over $\F_q$ with base-change $X_k=X\times_{\mathrm{Spec}(\F_q)}\mathrm{Spec}(k)$, the motivic congruence zeta function of $X$ coincides with the classical one. That is, we obtain the following equation (as an equality in $\Q[[t]]$):
\[
Z_\mot(X,t)=\exp\left(\sum_{n=1}^\infty\dfrac{\#(X(\F_{q^n}))}{n}t^n\right).
\]
\end{thm}
\begin{proof}
As in \cref{rmk:realization}, $\tr((F^\ast)^n|R^i\Hom_{\D_\mot(k)}(\Z_\mot[X_k],\Z)\widetilde{\otimes}\Q_\ell)$ can be considered as the classical trace of $H^i(X_{k,\text{\'{e}t}},\Q_\ell)$ by $(F_{X_k}^n)^\ast$ where $\ell$ is a prime number coprime to $p$. For any positive integer $n\in\mathbb{Z}_{>0}$, by the Grothendieck--Lefschetz trace formula (\cite{SGA}), the number $\#(X(\F_{q^n}))$ of $\F_{q^n}$-rational points can be computed by $\sum_{i=0}^\infty(-1)^i\Tr((F_{X_k}^\ast)^n|H^i(X_{k,\text{\'{e}t}},\Q_\ell))$. By \cref{rmk:computation of local term}, the alternating sum is computed by $\#_{F,n}(1_{X_{k,\arc}})$:
\[
\#_{F,n}(1_{X_{k,\arc}})=\sum_{i=0}^\infty(-1)^i\Tr((F_{X_k}^\ast)^n|H^i(X_{k,\text{\'{e}t}},\Q_\ell))=\#(X(\F_{q^n})).
\]
These computations above give the desired identification.
\end{proof}

\begin{rmk}
As in \cref{exa:geometric points}, Scholze gives some description of $\D_\mot(K)$ for an algebraically closed (non-discrete) Banach field $K$. Our strategy can be applied to them since, for main result (\cref{thm:classical}), our strategy mainly relies on categorical traces, and we only use the Grothendieck--Lefschetz trace formula to identify $\#(X(\F_{q^n}))$ with the alternating sum by the \emph{usual} trace on \'{e}tale cohomology. For example, in \cref{rmk:formal discussion}, we can take $M=\D_\mot(K)$ and $X=\D_\mot(Y)$ (where $Y$ denotes a “good" analytic space over $K$), and define $\#_{f,X}$ for a “suitable" endomorphism $f$.
\exend
\end{rmk}


\begin{thebibliography}{1}
\bibitem[Ber90]{B} Vladimir G. Berkovich, {\em Spectral theory and analytic geometry over non-Archimedean fields}, Math. Surv. Monogr.
\textbf{33}, AMS, 1990.
\bibitem[Cla25]{C} Dustin Clausen, {\em Duality and linearization for $p$-adic lie groups}, preprint (arXiv:\href{https://arxiv.org/abs/2506.18174}{2506.18174}), 2025.
\bibitem[CC]{CC} Dustin Clausen and Peter Scholze, {\em Condensed Mathematics and Complex Geometry}, preprint (arXiv:\href{https://arxiv.org/abs/2605.11731}{2605.11731}), 2026.
\bibitem[Gel41]{G} Israil M. Gelfand, {\em Normierte Ringe}, Sb. Math., \textbf{51} (1941) 1, 3-24.
\bibitem[SGA5]{SGA} Alexsander Grothendieck et al., {\em S\'{e}minaire de G\'{e}om\'{e}trie Alg\'{e}brique du Bois Marie - 1965-66 - Cohomologie l-adique et Fonctions L - (SGA 5)}, Lecture Notes in Mathematics, Vol. 589. Berlin; New York: Springer-Verlag, 1977.
\bibitem[HM24]{HM} Claudius Heyer and Lucas Mann, {\em 6-Functor Formalisms and Smooth Representations}, preprint (arXiv:\href{https://arxiv.org/abs/2410.13038}{2410.13038}), 2024.
\bibitem[HSS17]{HSS} Marc Hoyois, Sarah Scherotzke, and Nicolo Sibilla, {\em Higher traces, noncommutative motives, and the categorified chern character}, Adv. Math. \textbf{309} (2017), 97-154.
\bibitem[Kah24]{Ka} Bruno Kahn, {\em Zeta and L functions of Voevodsky motives}, preprint (arXiv:\href{https://arxiv.org/abs/2412.08437}{2412.08437}), 2024.
\bibitem[Sch24]{S} Peter Scholze, {\em Berkovich motives}, preprint (arXiv:\href{https://arxiv.org/abs/2412.03382}{2412.03382}), 2024, to appear in J. Amer. Math. Soc..
\bibitem[Wei49]{W} Andr\'{e} Weil, {\em Numbers of solutions of equations in finite fields}, Bull. Amer. Math. Soc. \textbf{55} (1949), 497–508.
\end{thebibliography}
\end{document}